\documentclass[12pt]{amsart}
\usepackage{amsmath,latexsym,amscd,amsbsy,amssymb,amsfonts,amsthm,fleqn,leqno,
euscript, graphicx}
\usepackage{color}
\usepackage{tikz}
\numberwithin{equation}{section}

\newtheorem{thm}{Theorem}[section]
\newtheorem{pro}[thm]{Proposition}
\newtheorem{lem}[thm]{Lemma}
\newtheorem{cor}[thm]{Corollary}

\theoremstyle{definition}

\theoremstyle{remark}
\newtheorem{claim}[thm]{Claim}

\begin{document}

\title[On homogeneity of Cantor cubes]
{On homogeneity of Cantor cubes}

\author{E. Shchepin}
\address{Steklov Mathematical Institute of Russian Academy of Sciences,
8 Gubkina St. Moscow, 119991, Russia}
\email{scepin@yandex.ru}

\author{V. Valov}
\address{Department of Computer Science and Mathematics,
Nipissing University, 100 College Drive, P.O. Box 5002, North Bay,
ON, P1B 8L7, Canada} \email{veskov@nipissingu.ca}

\thanks{The second author was partially supported by NSERC
Grant 261914-19.}




\begin{abstract}
We discuss the question of extending homeomorphism between closed subsets of the Cantor cube $D^\tau$. It is established that any homeomorphism between two closed negligible subset of $D^\tau$ can be extended to an autohomeomorphism of $D^\tau$.
\end{abstract}

\maketitle




\section{Introduction}

Knaster-Reichbach \cite{kr} established the following theorem (see also \cite{jp} for other types of zero-dimensional separable metric spaces): \textit{Let $X$ and $Y$ be compact, perfect zero-dimensional metric spaces, and let $P$ and $K$ be closed nowhere dense subsets of $X$ and $Y$, respectively. If $f$ is a homeomorphism between $P$ and $K$, then there exists a homeomorphism between $X$ and $Y$ extending $f$}.

If we omit the metrizability condition in the Knaster-Reichbach's theorem, then the conclusion is not anymore true. In order to obtain a correct generalization of the theorem, first of all, it is necessary to find the correct analogue of the condition "nowhere dense". Moreover, the perfectness condition  can be formulated as the nowhere density of the points.

Such an analogue is the following concept of negligibility. A subset of a topological space is called {\em negligible} if it does not contain a non-empty intersection of a family of open sets such that the cardinality of the family is less than the weight of the space.
Note that for  metric compacta the condition of nowhere density is equivalent to the condition of negligibility.

Now we able to present a generalization of the Knaster-Reichbach's theorem.
\begin{thm} \label{analogue} Let $X$ and $Y$ be
compact,  zero-dimensional absolute extensors for 0-dimensional spaces of the same weight with negligible points, and let $P$ and $K$ be closed
negligible subsets of $X$ and $Y$, respectively. If $f$ is a homeomorphism between
$P$ and $K$, then there exists a homeomorphism between
$X$ and $Y$ extending $f$.
\end{thm}
This theorem for metric compacta turns into Knaster-Reichbach's theorem, because every metric compact is an absolute extensor in dimension 0.
In general, it is very difficult to avoid the condition of being an absolute extensor in dimension 0, because before you can continue homeomorphisms, you need to be able to continue at least continuous maps.

Since the condition of the negligibility of a point in a compactum is equivalent to having a character at that point equal to the weight of the compactum, then Theorem 1  from \cite{sc} allows us to assert that the compacta $X$ and $Y$ in Theorem 1.1 are homeomorphic to the Cantor cube $D^\tau$, where $\tau$ is the weight of $X$ and $Y$. Therefore, the above theorem can be obtained from the following special case of its own.

\begin{thm}\label{cube}
Let $f$ be a homeomorphism between closed negligible subsets $P$ and $K$  of $D^\tau$,  then $f$ can be extended to a homeomorphism on $D^\tau$.
\end{thm}

Since every subset of the weight less than the weight of the entire space is negligible, the following consequence follows from the formulated theorem.

\begin{cor}
If $P,K$ are closed subsets of $D^\tau$ both of weight $<\tau$, then every homeomorphism between $P$ and $K$ can be extended to homeomorphism on $D^\tau$.
\end{cor}

Actually, we prove a more general version of Theorem \ref{cube}, see Theorem 4.1 below.

\section{Some preliminary results}
Anywhere below, by a homeomorphism we always mean a surjective homeomorphism.
We need a more precise notion of negligibility.
For a space $X$, a subset $P\subset X$ and an infinite cardinal $\lambda$ we denote by $P^{(\lambda)}$ the {\em $\lambda$-interior} of $P$ in $X$, i.e. the set all $x\in P$ such that there exists a $G_\lambda$-subset $K$ of $X$ with $x\in K\subset P$. If $\lambda$ is finite, then $P^{(\lambda)}$ is the ordinary interior of $P$ and it is denoted by $P^{(0)}$.  If there exists $\tau\geq\aleph_0$ with $P^{(\lambda)}$ is empty for all $\lambda<\tau$, we say that $P$ is {\em $\tau$-negligible in $X$}.

Let $X=\prod_{\alpha\in A}X_\alpha$ be a product of spaces and $B\subset A$. If $P\subset X$, then $P_B$ denotes $\pi_B(P)$, where
$\pi_B:X\to\prod_{\alpha\in B}X_\alpha$ is the projection.
\begin{pro}\label{countable}
Let $X=\prod_{\alpha\in A}X_\alpha$ be a product of compact metric spaces, $P$ and $K$ be closed subsets of $X$ and $f:P\to K$ be a homeomorphism. Then for any countable set $C\subset A$ there is a countable set $B\subset A$ and a homeomorphism $f_B:P_B\to K_B$ such that $C\subset B$ and $\pi_B\circ f=f_B\circ\pi_B$.
\end{pro}
\begin{proof} Obviously this is true for a countable set $A$, so we suppose $A$ is uncountable. Let $g: K\to P$ be the inverse of $f$.
Using that any continuous function on $X$ depends on countably many coordinates \cite{mi}, we construct by induction a sequence of countable sets $B(n)\subset A$ and maps $f_n:P_{B(n+1)}\to K_{B(n)}$ and $g_n:K_{B(n+1)}\to P_{B(n)}$ such that:
\begin{itemize}
\item $B(1)=C$ and $B(n)\subset B(n+1)$;
\item $\displaystyle\pi_{B(n)}\circ f=f_n\circ\pi_{B(n+1)}$;
\item $\displaystyle\pi_{B(n)}\circ g=g_n\circ\pi_{B(n+1)}$;
\end{itemize}
Then $B=\bigcup_{n=1}^\infty B(n)$ is countable and there exist maps $f_B:P_B\to K_B$ and $g_B:K_B\to P_B$ with
$\displaystyle\pi_B\circ f=f_B\circ\pi_B$ and $\displaystyle\pi_B\circ g=g_B\circ\pi_B$. Then $f_B$ is a homeomorphism between $P_B$ and $K_B$.
\end{proof}
 In the situation of Proposition \ref{countable}, a subset $B\subset A$ is called {\em $f$-admissible} if there exists a homeomorphism $f_B:P_B\to K_B$ with
 $\displaystyle\pi_{B}\circ f=f_B\circ\pi_{B}$. It easily seen that arbitrary union of $f$-admissible sets is also $f$-admissible.

In \cite{m1} $\tau$-negligible sets  were considered  under the name $\widetilde{G}_\tau$-sets.
By \cite[Lemma 6]{m1}, if $X$ is a product of metric compacta, then $F$ is $\tau$-negligible in $X$ if and only if the $\pi$-character $\pi\chi(F,X)$
of $F$ in $X$ is $\geq\tau$.
Recall that $\pi\chi(F,X)$  is the smallest cardinality $\lambda$ such that there is an open family $\mathcal U$  in $X$ of cardinality $\lambda$ with the following property: every neighborhood of $F$ in $X$ contains an element of $\mathcal U$.

Next lemma is a modification of \cite[Theorem 2]{m1}.
\begin{lem}
Let $X=\prod_{\alpha\in A}X_\alpha$ be a product of compact metric spaces and $P$ be a closed set in $X$.
Suppose $C\subset A$ is a set of cardinality $<\tau$ such that $(\{z\}\times X_{A\backslash C})\cap P$ is $\tau$-negligible in $\{z\}\times X_{A\backslash C}$ for every $z\in P_C$. Then $P_{A\backslash C}\neq X_{A\backslash C}$.
\end{lem}
\begin{proof}
Indeed, otherwise we may assume that the projection $\pi_{A\backslash C}$ restricted on $P$  is an irreducible map onto $X_{A\backslash C}$, and denote this map by $f$.
Consequently, $$\pi\chi(f((\{z\}\times X_{A\backslash C})\cap P),X_{A\backslash C})\leq\pi\chi((\{z\}\times X_{A\backslash C})\cap P,P)$$
for every $z\in P_C$. On the other hand, $\pi\chi((\{z\}\times X_{A\backslash C})\cap P,P)\leq\pi\chi(z,P_C)<\tau$. So,
$\pi\chi(f((\{z\}\times X_{A\backslash C})\cap P),X_{A\backslash C})<\tau$. That, according to \cite[Lemma 6]{m1}, contradicts the fact that $f((\{z\}\times X_{A\backslash C})\cap P)$ is $\tau$-negligible in $X_{A\backslash C}$.
\end{proof}
\begin{cor}
Let $X=\prod_{\alpha\in A}X_\alpha$ be a product of compact metric spaces and $P$ be a closed $\tau$-negligible set in $X$.
If $C\subset A$ is a set of cardinality $<\tau$,
then there is a set $B\subset A$ containing $C$ such that $B\backslash C$ is countable and $\pi_{B\backslash C}(P)$ is nowhere dense in $X_{B\backslash C}$. If, in addition, $f:P\to K$ is a homeomorphism with $K\subset X$ and $C$ is $f$-admissible, then we can assume that $B$ is also $f$-admissible.
\end{cor}
\begin{proof}
Since $P$ is a $\tau$-negligible set in $X$, so are the sets $P(z)=(\{z\}\times X_{A\backslash C})\cap P$ for all $z\in P_C$. This implies that each $\pi_{A\backslash C}(P(z))$ is $\tau$-negligible in $X_{A\backslash C}$. Indeed, otherwise $\pi_{A\backslash C}(P(z^*))$ would contain a closed $G_\lambda$-set in $X_{A\backslash C}$ for some $z^*\in P_C$ and $\lambda<\tau$. Because $\{z^*\}\times X_{A\backslash C}$ is $G_\mu$-set in $X$, where $\mu$ is the cardinality of $C$, $P(z^*)$ contains a $G_{\lambda'}$-subset of $X$ with $\lambda'=\max\{\lambda,\mu\}<\tau$, a contradiction. Therefore,
we can apply Lemma 2.2 countably many times to construct by induction a disjoint sequence $\{C_n\}$ of finite subsets of $A\backslash C$ such that
\begin{itemize}
\item $C_1\subset A\backslash C$;
\item $C_{n+1}\subset A\backslash\bigcup_{k\leq n}C\cup C_k$;
\item $\pi_{C_n}(P)\neq X_{C_n}$ for all $n$.
\end{itemize}
Then $B=\bigcup_{n\geq 1}C\cup C_n$ is the required set. If $f:P\to K$ is a homeomorphism and $C$ is $f$-admissible, then for every $\alpha\in A$ fix a countable $f$-admissible set $B(\alpha)$ containing $\alpha$, see Proposition 2.1. Next, using Lemma 2.2 we construct
a disjoint sequence $\{C_n\}$ of finite sets with
\begin{itemize}
\item $C_1\subset A\backslash C$;
\item $C_{n+1}\subset A\backslash\bigcup_{k\leq n}C\cup C_k'$, where $C_k'=\bigcup_{\alpha\in C_k}B(\alpha)$;
\item $\pi_{C_n}(P)\neq X_{C_n}$ for all $n$.
\end{itemize}
Then $B=\bigcup_{n\geq 1}C\cup C_n'$ is $f$-admissible and satisfies the required conditions.
\end{proof}

\section{Homeomorphisms on product spaces}

For any space $X$ let $\mathcal H(X)$ denote the set of all homeomorphisms on $X$ equipped with the compact-open topology. In this section we prove that $\mathcal H(X)$, where $X$ is a product of compact metric spaces, is an absolute extensor for compact 0-dimensional spaces. Our proof uses the technique developed in \cite{m}.

\begin{pro}\label{hom}
Let $K\subset\mathcal H(X)$ be a Lindel$\ddot{o}$f subset, where $X=\prod_{\alpha\in A}X_\alpha$ is a product of compact metric spaces with $|A|=\tau$. Then
$A$ can be covered by a family of sets $\{A(\alpha):\alpha\in\omega(\tau)\}$ such that for every $\alpha$ we have:
\begin{itemize}
\item $A(\alpha)=\bigcup_{\gamma<\alpha}B(\gamma)$ if $\alpha$ is a limit ordinal;
\item $A(\alpha)\subset A(\alpha+1)$ and $A(\alpha+1)\backslash A(\alpha)$ is countable for all $\alpha$;
\item For every $f\in K$ there is $f_\alpha\in\mathcal H(X_{A(\alpha)})$ with $\pi_{A(\alpha)}\circ f=f_\alpha\circ\pi_{A(\alpha)}$.
\end{itemize}
\end{pro}
\begin{proof}
In case $A$ is countable we take $A(\alpha)=A$ for all $\alpha$. Suppose $A$ is uncountable and let $B\subset A$ be a countable set. Take a sequence of open covers $\{\mathcal U_n\}_{n\geq 1}$ of $X_B$ such that $\rm{diam}(\mathcal U_n)<1/n$ for all $n$. Since $X_B$ is metrizable and $X$ is compact, the compact-open topology on the function space $C(X,X_B)$ coincides with the limitation topology \cite{to}. Recall that \cite{bo1}
$U\subset C(X,X_B)$ is open with respect to the limitation topology if for every $f\in U$ there is $\mathcal V\in\rm{cov}(X_B)$  such that $U$ contains the set
$B(f,\mathcal V)=\{g\in C(X,X_B):g{~}\mbox{is}{~}\mathcal V-\mbox{close to}{~}f\}.$ Here, $\rm{cov}(X_B)$ is the family of all open covers of $X_B$ and $g$ is $\mathcal V$-close $f$ provided for any $x\in X$ there is $V\in\mathcal V$ containing both points $f(x)$ and $g(x)$. In particular, every $B(f,\mathcal V)$ contains a neighborhood $B_*(f,\mathcal V)$ of $f$. There exists a natural map $p_B:\mathcal H(X)\to C(X,X_B)$,
$p_B(h)=\pi_B\circ h$, which is continuous when both $\mathcal H(X)$ and $C(X,X_B)$ carry the compact-open topology.

\begin{claim}\label{1}
There is a countable set $\Gamma(B)\subset A$ containing $B$ such that for every $f\in K$ there exist
$f_{\Gamma(B)},g_{\Gamma(B)}\in C(X_{\Gamma(B)},X_B)$ with $\pi_{B}\circ f=f_{\Gamma(B)}\circ\pi_{\Gamma(B)}$ and
$\pi_{B}\circ f^{-1}=g_{\Gamma(B)}\circ\pi_{\Gamma(B)}$.
\end{claim}
Since for each $n$ the family $\{B_*(\pi_{B}\circ f,\mathcal U_n):f\in K\}$ is an open cover of $p_B(K)$, there is a sequence $\{f_{ni}\}_{i\geq 1}\subset\mathcal H(X)$ such that
$\{B_*(\pi_{B}\circ f_{ni},\mathcal U_n):i\geq 1\}$ covers $p_B(K)$. Similarly, there exists sequences $\{g_{ni}\}_{i\geq 1}\subset\mathcal H(X)$, $n\geq 1$, such that
$\{B_*(\pi_{B}\circ g_{ni},\mathcal U_n):i\geq 1\}$ covers $p_B(K^{-1})$, where $K^{-1}=\{f^{-1}:f\in K\}$.
As in the proof of Proposition \ref{countable}, there is a countable set $\Gamma(B)$ containing $B$ and corresponding to each $n$ sequences $\{\varphi_{ni}\}_{i\geq 1}\subset C(X_{\Gamma(B)},X_B)$ and $\{\phi_{ni}\}_{i\geq 1}\subset C(X_{\Gamma(B)},X_B)$
such that $\pi_{B}\circ f_{ni}=\varphi_{ni}\circ\pi_{\Gamma(B)}$ and $\pi_{B}\circ g_{ni}=\phi_{ni}\circ\pi_{\Gamma(B)}$
for all $n,i$. Then for every $f\in K$ there exist two sequences
$\displaystyle\{\varphi_{ni_n}\}_{n\geq 1}\subset C(X_{\Gamma(B)},X_B)$ and $\displaystyle\{\phi_{ni_n}\}_{n\geq 1}\subset C(X_{\Gamma(B)},X_B)$
such that $\pi_{B}\circ f$ is $\mathcal U_n$-close to $\displaystyle\varphi_{ni_n}\circ\pi_{\Gamma(B)}$ and
$\pi_{B}\circ f^{-1}$ is $\mathcal U_n$-close to $\displaystyle\phi_{ni_n}\circ\pi_{\Gamma(B)}$
for each $n$. The last condition implies that the sequence $\displaystyle\{\varphi_{ni_n}\}_{n\geq 1}$ uniformly converges to a map $f_{\Gamma(B)}\in C(X_{\Gamma(B)},X_B)$ and $\pi_{B}\circ f=f_{\Gamma(B)}\circ\pi_{\Gamma(B)}$. The sequence
$\displaystyle\{\phi_{ni_n}\}_{n\geq 1}$ also converges uniformly to a map $g_{\Gamma(B)}\in C(X_{\Gamma(B)},X_B)$ and
$\pi_{B}\circ f^{-1}=g_{\Gamma(B)}\circ\pi_{\Gamma(B)}$.

\begin{claim}\label{2}
For every countable set $B\subset A$ there is a countable set $\Lambda(B)\subset A$ containing $B$ such that for every $f\in K$ there exist
homeomorphisms
$f_{\Lambda(B)},g_{\Lambda(B)}\in \mathcal H(X_{\Lambda(B)})$ with $\pi_{\Lambda(B)}\circ f=f_{\Lambda(B)}\circ\pi_{\Lambda(B)}$ and
$\pi_{\Lambda(B)}\circ f^{-1}=g_{\Lambda(B)}\circ\pi_{\Lambda(B)}$.
\end{claim}
Indeed, using the notations from Claim \ref{1}, we construct an increasing sequence $B(n)$ of countable subsets of $A$ such that $B(0)=B$ and
$B(n)=\Gamma(B(n-1))$ for every $n\geq 1$. Let $\Lambda(B)=\bigcup_{n\geq 1}B(n)$. Then for every $f\in K$ there exist sequences
$\{f_n\}$ and $\{g_n\}$ satisfying the following conditions:
\begin{itemize}
\item $f_n,g_n\in C(X_{B(n)},X_{B(n-1)})$;
\item $\pi_{B(n-1)}\circ f=f_n\circ\pi_{B(n)}$ and $\pi_{B(n-1)}\circ f^{-1}=g_n\circ\pi_{B(n)}$.
\end{itemize}
The last two conditions imply that if $f\in K$, then for every $x,y\in X$ with $\pi_{\Lambda(B)}(x)=\pi_{\Lambda(B)}(y)$ we have
$\pi_{\Lambda(B)}(f(x))=\pi_{\Lambda(B)}(f(y))$ and $\pi_{\Lambda(B)}(f^{-1}(x))=\pi_{\Lambda(B)}(f^{-1}(y))$. Therefore, there exist homeomorphisms
$f_{\Lambda(B)},g_{\Lambda(B)}\in \mathcal H(X_{\Lambda(B)})$ satisfying the required conditions.

Now, for every $\alpha<\omega(\tau)$ take a countable set $\Lambda(\alpha)\subset A$ satisfying the hypotheses of Claim \ref{2} with  $B=\{\alpha\}$.
Let $A(\alpha)=\bigcup_{\gamma<\alpha}\Lambda(\gamma)$ if $\alpha$ is a limit ordinal, and $A(\alpha)=A(\alpha-1)\cup\Lambda(\alpha)$ otherwise. Since for every $\alpha\in A$ and $f\in K$ there exist $f_{\Lambda(\alpha)},g_{\Lambda(\alpha)}\in \mathcal H(X_{\Lambda(\alpha)})$ with $\pi_{\Lambda(\alpha)}\circ f=f_{\Lambda(\alpha)}\circ\pi_{\Lambda(\alpha)}$ and
$\pi_{\Lambda(\alpha)}\circ f^{-1}=g_{\Lambda(\alpha)}\circ\pi_{\Lambda(\alpha)}$, we have $\pi_{A(\alpha)}(f(x))=\pi_{A(\alpha)}(f(y))$ and
$\pi_{A(\alpha)}(f^{-1}(x))=\pi_{A(\alpha)}(f^{-1}(y))$ for any pair $x,y\in X$ with $\pi_{A(\alpha)}(x)=\pi_{A(\alpha)}(y)$. This  yields a homeomorphism $f_\alpha\in\mathcal H(X_{A(\alpha)})$ such that $\pi_{A(\alpha)}\circ f=f_\alpha\circ\pi_{A(\alpha)}$.
\end{proof}

Next theorem is an analogue of Mednikov's result \cite[Corollary 3]{m} stating that $\mathcal H([0,1]^A)$ is an absolute extensor for compact spaces.
\begin{thm}
Let $X=\prod_{\alpha\in A}X_\alpha$ be a product of compact metric spaces. Then $\mathcal H(X)$ is an absolute extensor
for zero-dimensional compact spaces.
\end{thm}
\begin{proof}
Suppose $Y$ is a 0-dimensional compact space and $g:P\to \mathcal H(X)$ be a map, where $P$ is closed in $Y$. If $A$ is countable, then
$\mathcal H(X)$ is a complete separable metric space. So, we can consider $\mathcal H(X)$ as a closed subset of $l_2$ and find a continuous extension $\widetilde g:Y\to l_2$ of $g$. By the well-known factorization theorem, there is a metric 0-dimensional compactum $Y_0$ and a maps $g_0:Y_0\to l_2$ and $h:Y\to Y_0$ with $\widetilde g=g_0\circ h$. Then $P_0=h(P)$ is a closed subset of $Y_0$. Since $Y_0$ is $0$-dimensional, there is a retraction $r:Y_0\to P_0$. Finally, the composition $(g_0|P_0)\circ r\circ h$ is a map from $Y$ to $\mathcal H(X)$ extending $g$.

Assume $A$ is an uncountable set of cardinality $\tau$. Then $A$ can be covered by a family $\xi=\{A(\alpha):\alpha\in\omega(\tau)\}$ satisfying the hypotheses of Proposition \ref{hom} with $K=g(P)$.
Then for every $\alpha\in\omega(\tau)$ and $f\in K$ we have
\begin{itemize}
\item[(*)] $\pi^{A(\alpha+1)}_{A(\alpha)}\circ f_{\alpha+1}=f_\alpha\circ\pi^{A(\alpha+1)}_{A(\alpha)}$.
\end{itemize}
Denote by $\mathcal H_\xi(X)$ the subspace of $\mathcal H(X)$ consisting of all $f$ with the following property: For every $\alpha$ there is
$f_\alpha\in\mathcal H(X_{A(\alpha)})$ such that $f_\alpha$ and $f_{\alpha+1}$ satisfy $(*)$. Therefore, $K\subset \mathcal H_\xi(X)$. So, it suffice to show that $\mathcal H_\xi(X)$ is an absolute extensor for 0-dimensional compacta.

Since
$X_{A(\alpha+1)}=X_{A(\alpha)}\times X_{A(\alpha+1)\backslash A(\alpha)}$, the map $f_{\alpha+1}$ is of the form
$f_{\alpha+1}(x,y)=(f_{\alpha}(x),g(x,y))$, where $x\in X_{A(\alpha)}$ and $y\in X_{A(\alpha+1)\backslash A(\alpha)}$ and $g$ is a map
from $X_{A(\alpha+1)}$ into $X_{A(\alpha+1)\backslash A(\alpha)}$ such that for any $x\in X_{A(\alpha)}$ the map
$\varphi_g(x)$, $\varphi_g(x)(y)=g(x,y)$, is a homeomorphism on $X_{A(\alpha)\backslash A(\alpha)}$. Therefore, by \cite[Theorem 3.4.9]{en},
the correspondence $\varphi_g\leftrightarrow g$ is a homeomorphism between $C(X_{A(\alpha)},\mathcal H(X_{A(\alpha)\backslash A(\alpha)}))$ and
the subset of $C(X_{A(\alpha+1)},X_{A(\alpha+1)\backslash A(\alpha)})$ consisting of all $g$ such that for each $x\in X_{A(\alpha)}$ the map $\varphi_g(x)$ belongs to $\mathcal H(X_{A(\alpha)\backslash A(\alpha)})$. Hence, the correspondence $(f_\alpha,\varphi_g)\leftrightarrow f_{\alpha+1}$ provides a homeomorphism between the spaces $C(X_{A(\alpha)},\mathcal H(X_{A(\alpha)\backslash A(\alpha)}))$ and $\mathcal H_\alpha(X_{A(\alpha+1)})$, where $\mathcal H_\alpha(X_{A(\alpha+1)})$ consists of all homeomorphisms $f_{\alpha+1}$ on $X_{A(\alpha+1)}$ satisfying equality $(*)$. This means that there is one-to-one correspondence between $\mathcal H_\xi(X)$ and the product
$\mathcal H(X_{A(0)})\times\Pi_{\alpha<\omega(\tau)}C(X_{A(\alpha)},\mathcal H(X_{A(\alpha+1)\backslash A(\alpha)}))$. This correspondence is a homeomorphism when all function spaces carry the compact-open topology.

It remains to show that each multiple in this product is an absolute extensor for 0-dimensional compacta. This is true for $\mathcal H(X_{A(0)})$ because
$A(0)$ is countable. Let's prove that each $C(X_{A(\alpha)},\mathcal H(X_{A(\alpha+1)\backslash A(\alpha)}))$ is also an absolute extensor for 0--dimensional compacta. To this end,
take a pair $L\subset Z$ of 0-dimensional compacta and a map
$\theta$ from $L$ to $C(X_{A(\alpha)},\mathcal H(X_{A(\alpha+1)\backslash A(\alpha)}))$. Since $\mathcal H(X_{A(\alpha+1)\backslash A(\alpha)})$ is a separable complete metric space and $X_{A(\alpha)}$ is a compactum of weight $\leq\tau$, $C(X_{A(\alpha)},\mathcal H(X_{A(\alpha+1)\backslash A(\alpha)}))$ is a complete metric space of weight $\leq\tau$. So, $C(X_{A(\alpha)},\mathcal H(X_{A(\alpha+1)\backslash A(\alpha)}))$ can be embedded in $l_2(\tau)$ and $\theta$ can be extended to a map $\eta:Z\to l_2(\tau)$. Because the image
$\eta(Z)$ is a metric compactum, we can repeat the arguments from the first paragraph of our proof and obtain an extension
$\widetilde\theta:Z\to C(X_{A(\alpha)},\mathcal H(X_{A(\alpha+1)\backslash A(\alpha)}))$ of $\theta$.
\end{proof}

Everywhere below by $\mathfrak{C}$ we denote the Cantor set.
\begin{cor}
Let $P$ and $K$ be proper closed subsets of $\mathfrak{C}^A$ and $f$ be a homeomorphism between $P$ and $K$. Suppose there exist a proper subset $B\subset A$ and a homeomorphism $f_B$ between $P_B$ and $K_B$ such that
\begin{itemize}
\item $P=P_{B}\times\mathfrak{C}^{A\backslash B}$ and $K=K_{B}\times\mathfrak{C}^{A\backslash B}$;
\item $f_B\circ\pi_B=\pi_B\circ f$;
\item $f_B$ can be extended to a homeomorphism $\widetilde f_B\in\mathcal H(\mathfrak{C}^B)$;
\end{itemize}
Then $f$ can be extended to a homeomorphism $\widetilde f\in\mathcal H(\mathfrak{C}^A)$.
\end{cor}
\begin{proof}
Since $f_B\circ\pi_B=\pi_B\circ f$, $f$ is of the form $f(x,y)=(f_B(x),h(x,y))$ with $(x,y)\in P_{B}\times\mathfrak{C}^{A\backslash B}$ such that for each $x\in P_{B}$ the map $\varphi_x$, defined by $\varphi_x(y)=h(x,y)$, belongs to
$\mathcal H(\mathfrak{C}^{A\backslash B})$. So, we have a map $\varphi:P_{B}\to\mathcal H(\mathfrak{C}^{A\backslash B})$, see \cite[Theorem 3.4.9]{en}. By Theorem 3.4, we can extend $\varphi$ to a map $\widetilde\varphi: \mathfrak{C}^B\to\mathcal H(\mathfrak{C}^{A\backslash B})$ and define $\widetilde h:\mathfrak{C}^A\to \mathfrak{C}^{A\backslash B}$,
$\widetilde h(x,y)=\widetilde\varphi(x)(y)$, where $(x,y)\in \mathfrak{C}^B\times \mathfrak{C}^{A\backslash B}$. Finally, $\widetilde f(x,y)=(\widetilde f_B,\widetilde h(x,y))$ provides a homeomorphism in $\mathcal H(\mathfrak{C}^A)$ extending $f$.
\end{proof}

\section{Extending homeomorphisms}
In this section we prove the following generalization of Theorem \ref{cube}.
\begin{thm}
Let $f$ be a homeomorphism between closed $\lambda$-negligible subsets $P$ and $K$  of $D^\tau$ with $\lambda\leq\tau$ such that $P=P^{(\lambda)}$ and $K=K^{(\lambda)}$. Then $f$ can be extended to a homeomorphism on $D^\tau$.
\end{thm}

\begin{lem}
Let $X, Y$ be 0-dimensional paracompact spaces and $\mathfrak{C}$ be the Cantor set. Suppose $P'\subset X\times\mathfrak{C}$, $K'\subset Y\times\mathfrak{C}$ are closed sets such that $\pi_X(P')=X$, $\pi_{Y}(K')=Y$ and the sets
$\pi_{\mathfrak{C}}((\{x\}\times\mathfrak{C})\cap P')$ and $\pi_{\mathfrak{C}}((\{y\}\times\mathfrak{C})\cap K')$ are nowhere dense in $\mathfrak{C}$ for all $x\in X$ and $y\in Y$.
Let $f:P'\to K'$ and $g:X\to Y$ be homeomorphisms with $g\circ\pi_X=\pi_{Y}\circ f$.
Then $f$ can be extended to a homeomorphism
$\widetilde f:X\times\mathfrak{C}\to Y\times\mathfrak{C}$ such that $g\circ\pi_X=\pi_{Y}\circ\widetilde f$.
\end{lem}
\begin{proof}
For any $x\in X$ let $\Phi(x)$ be the set of all $h\in\mathcal H(\mathfrak{C})$ such that $f(x,c)=(g(x),h(c))$ for all $c\in\pi_X^{-1}(x)\cap P'$. Since for every $x\in X$ the restriction $f|\pi_X^{-1}(x)\cap P'$ is a homeomorphism between the nowhere dense subsets $\pi_X^{-1}(x)\cap P'$ and $\pi_{Y}^{-1}(g(x))\cap K'$ of $\mathfrak{C}$, by \cite{kr}, there is $h_x\in\mathcal{H}(\mathfrak{C})$ extending $f|\pi_X^{-1}(x)\cap P'$. Hence,
$\Phi(x)\neq\varnothing$.
Moreover, the sets $\Phi(x)$ are closed in $\mathcal H(\mathfrak{C})$ equipped with the compact-open topology. So, we have a set-valued map
$\Phi:X\rightsquigarrow\mathcal H(\mathfrak{C})$. One can show that if $\Phi$ admits a continuous selection $\phi:X\to\mathcal H(\mathfrak{C})$,
then the map $\widetilde f:X\times\mathfrak{C}\to Y\times\mathfrak{C}$, defined by $\widetilde f(x,c)=(g(x),\phi(x)(c))$, is the required homeomorphism extending $f$. Therefore, according to Michael's \cite{em} zero-dimensional selection theorem, it suffices to show that $\Phi$ is lower semi-continuous.

To prove that, let $x^*\in X$ be a fixed point and $h^*\in\Phi(x^*)\cap W$, where $W$ is open in $\mathcal H(\mathfrak{C})$. We can assume that
$W$ is of the form $\{h\in\mathcal H(\mathfrak{C}): h(U_i)\subset V_i, i=1,2,..k\}$, where $\{U_i\}_{i=1}^k$ is a clopen disjoint cover of
$\{x^*\}\times\mathfrak{C}$ and $\{V_i\}_{i=1}^k$ is a disjoint clopen cover of $\{g(x^*)\}\times\mathfrak{C}$. We extend the sets $U_i$ and $V_i$ to clopen sets $\widetilde U_i\subset X\times\mathfrak{C}$ and $\widetilde V_i\subset Y\times\mathfrak{C}$ such that
\begin{itemize}
\item[(1)] $\widetilde U_i=O(x^*)\times U_i$ and $\widetilde V_i=g(O(x^*))\times V_i$, where $O(x^*)$ is a clopen neighborhood of $x^*$ in $X$;
\item[(2)] $O(x^*)$ can be chosen so small that $f(\widetilde U_i\cap P')\subset\widetilde V_i\cap K'$.
\end{itemize}
We are going to show that for every $x\in O(x^*)$ there exists $h_x\in\Phi(x)\cap W$. We fix such $x$ and observe that all sets
$\widetilde U_i(x)=\widetilde U_i\cap(\{x\}\times\mathfrak{C})$ and $\widetilde V_i(x)=\widetilde V_i\cap(\{g(x)\}\times\mathfrak{C})$ are compact and perfect. Moreover, $P_i'(x)=\widetilde U_i(x)\cap P'$ and  $K_i'(x)=\widetilde V_i\cap K'$ are nowhere dense sets in $\widetilde U_i(x)$
and $\widetilde V_i(x)$, respectively, and $f^x_i=f|(\widetilde U_i(x)\cap P')$ is a homeomorphism between $P_i'(x)$ and $K_i'(x)$ for every $i$.
Hence, by \cite{kr}, there exist homeomorphisms $\widetilde f^x_i:\widetilde U_i(x)\to\widetilde V_i(x)$ extending $f^x_i$. Because
$\{\widetilde U_i(x)\}_{i=1}^k$ and $\{\widetilde V_i(x)\}_{i=1}^k$ are disjoint clopen covers of $\pi_X^{-1}(x)$ and
$\pi_Y^{-1}(g(x))$, respectively, the homeomorphisms $\widetilde f^x_i$, $i=1,2,..,k$, provide a homeomorphism $h_x'$ between $\pi_X^{-1}(x)$ and
$\pi_Y^{-1}(g(x))$ extending $f|\pi_X^{-1}(x)\cap P'$. Then the equality $h_x(c)=h_x'(x,c)$, $c\in\mathfrak{C})$, defines a homeomorphism
from $\mathcal{H}(\mathfrak{C})$ with $h_x\in\Phi(x)\cap W$. Therefore, $\Phi$ is lower semi-continuous.
\end{proof}

{\em Proof of Theorem $4.1$.}
We identify $D^\tau$ with $D^A$, where $A$ is a set of cardinality $\tau$.
We first consider the case when $\lambda=\tau$.
We already observed that the theorem is true when $A$ is countable.
So, let $A=\{\alpha:\alpha<\omega(\tau)\}$ be uncountable and consider two copies of $D^A=\mathfrak{C}^A$, one containing $P$ and the other containing $K$. Each of the open sets $\mathfrak{C}^A\setminus P$ and $\mathfrak{C}^A\setminus K$ contains dense functionally open sets, say $V(P)$ and $V(K)$, and choose a countable set $C\subset A$ such that $\pi_C^{-1}(\pi_C(V(P)))=V(P)$ and $\pi_C^{-1}(\pi_C(V(K)))=V(K)$. Hence, $P_B$ and $K_B$ are nowhere dense subsets of $\mathfrak{C}^B$ for any set $B\subset A$ containing $C$. Next,
using Corollary 2.3, we can cover $A$ by an increasing family $\{A(\alpha):\alpha<\omega(\tau)\}$ and find homeomorphisms $f_\alpha:P_{A(\alpha}\to K_{A(\alpha)}$ satisfying the following conditions:
\begin{itemize}
\item[(3)] $A_\alpha=\bigcup_{\beta<\alpha}A(\beta)$ if $\alpha$ is a limit ordinal and the cardinality of each $A(\alpha)$ is less than $\tau$;
\item[(4)] $A(\alpha+1)\backslash A(\alpha)$ is countable and $C\subset A(\alpha)$ for all $\alpha$;
\item[(5)] $\displaystyle\pi_{A(\alpha)}\circ f=f_\alpha\circ\pi_{A(\alpha)}$;
\item[(6)] $\pi_{A(\alpha+1)\backslash A(\alpha)}(P)$ and $\pi_{A(\alpha+1)\backslash A(\alpha)}(K)$ are nowhere dense sets in $\mathfrak{C}^{A(\alpha+1)\backslash A(\alpha)}$.
\end{itemize}

It remains to prove that each $f_\alpha$ can be extended to a homeomorphism $\widetilde f_\alpha\in\mathcal H(\mathfrak{C}^{A(\alpha)})$ such that
\begin{itemize}
\item[(7)] $\displaystyle\pi_{A(\alpha)}^{A(\alpha+1)}\circ\widetilde f_{\alpha+1}=\widetilde
f_\alpha\circ\pi_{A(\alpha)}^{A(\alpha+1)}$.
\end{itemize}
The proof is by transfinite induction.
The first extension $\widetilde f_1$ exists by \cite{kr} because $P_{A(1)}$ and $K_{A(1)}$ are nowhere dense subsets of $\mathfrak{C}^{A(1)}$.
If $\widetilde f_\alpha$ is already defined for all $\alpha<\beta$, where $\beta$ is a limit ordinal, then item (3) implies the existence of
$\widetilde f_\beta$. Therefore, we need only to define $\widetilde f_{\alpha+1}$ provided $\widetilde f_\alpha$ exists.

To this end, consider
the spaces $X=P_{A(\alpha)}$ and $Y=K_{A(\alpha)}$, the sets
$P'=P_{A(\alpha+1)}\subset X\times\mathfrak{C}^{A(\alpha+1)\backslash A(\alpha)}$,
$K'=K_{A(\alpha+1)}\subset Y\times\mathfrak{C}^{A(\alpha+1)\backslash A(\alpha)}$ and the homeomorphisms $f_{\alpha+1}$, $f_\alpha$.
For any $x\in X$ and $y\in Y$ let
$$P'(x)=P'\cap(\{x\}\times\mathfrak{C}^{A(\alpha+1)\backslash A(\alpha)}){~} \mbox{and}{~} K'(y)=K'\cap(\{y\}\times\mathfrak{C}^{A(\alpha+1)\backslash A(\alpha)}).$$

Item (6) yields that $\pi_{A(\alpha+1)\backslash A(\alpha)}(P'(x))$ and $\pi_{A(\alpha+1)\backslash A(\alpha)}(K'(y))$ are nowhere dense sets in $\mathfrak{C}^{A(\alpha+1)\backslash A(\alpha)}$ for every $x\in P_{A(\alpha)}$ and $y\in K_{A(\alpha)}$. Therefore, by Lemma 4.2, the homeomorphism $f_{\alpha+1}$ can be extended to a fiberwise homeomorphism
$$f_{\alpha+1}':P_{A(\alpha)}\times\mathfrak{C}^{A(\alpha+1)\backslash A(\alpha)}\to K_{A(\alpha)}\times\mathfrak{C}^{A(\alpha+1)\backslash A(\alpha)}.$$ Finally, by Corollary 3.5, there is a homeomorphism $\widetilde f_{\alpha+1}\in\mathcal H(\mathfrak{C}^{A(\alpha+1)})$ satisfying
condition (7).

Suppose now that $P$ and $K$ are $\lambda$-negligible in $\mathfrak{C}^A$ and  $\lambda<\tau$. Then $A$ is uncountable and $P=P^{(\lambda)}$, $K=K^{(\lambda)}$.  So, by \cite{ef}, there is a set $B\subset A$ of cardinality $\lambda$ with $P=\pi_B^{-1}(P_B)$ and $K=\pi_B^{-1}(K_B)$. According to Proposition 2.1, we can assume that $B$ is $f$-admissible. So, there exists a homeomorphism $f_B:P_B\to K_B$ with $\pi_B\circ f=f_B\circ\pi_B$. Moreover, the sets $P_B$ and $K_B$ are $\lambda$-negligible in $\mathfrak{C}^B$. Hence, by the already established case of Theorem 4.1, $f_B$ can be extended to a homeomorphism $\widetilde f_B\in\mathcal H(\mathfrak{C}^B)$. Then, Corollary 3.5 implies the existence of a homeomorphism
$\widetilde f\in\mathcal H(\mathfrak{C}^A)$ extending $f$. $\Box$


\end{document}